\documentclass [oneside,11pt] {amsart} 

\usepackage{amsfonts}
\usepackage{amssymb}
\usepackage{mathtools}
\usepackage{hyperref}
\usepackage{comment}
\usepackage{color,xcolor}
\usepackage{a4}
\usepackage{geometry}
\usepackage{todonotes}

\numberwithin{equation}{section}
\newtheorem{thm}{Theorem}[section]

\newtheorem{prop}{Proposition}[section]
\newtheorem{lemma}{Lemma}[section]

\newtheorem{remark}{Remark}[section]

\newcommand{\R}{\mathbb{R}}
\newcommand{\p}{\partial}
\renewcommand{\S}{\mathbb{S}}

\newcommand{\LV}{\left|}
\newcommand{\RV}{\right|}
\newcommand{\LC}{\left(}
\newcommand{\RC}{\right)}

\usepackage[framemethod=tikz]{mdframed}
\usepackage{url}
\definecolor{mycolor}{rgb}{0.122, 0.435, 0.698}
\definecolor{aliceblue}{rgb}{0.94, 0.97, 1.0}

\newmdenv[innerlinewidth=0.5pt, roundcorner=4pt,linecolor=mycolor,innerleftmargin=6pt,
innerrightmargin=6pt,innertopmargin=6pt,innerbottommargin=6pt]{mybox}
\newmdenv[backgroundcolor=aliceblue,innerlinewidth=0.5pt, roundcorner=4pt,linecolor=mycolor,innerleftmargin=6pt,
innerrightmargin=6pt,innertopmargin=6pt,innerbottommargin=6pt]{mybox1}

\begin{document}

\title[]{Stable determination of time-dependent collision kernel in the nonlinear Boltzmann equation}

\author[Lai]{Ru-Yu Lai}
\address
{R.-Y. Lai, School of Mathematics\\
	University of Minnesota\\ 
	Minneapolis, MN 55455, USA }

\email{rylai@umn.edu}

\author[Yan]{Lili Yan}

\address
{L. Yan, School of Mathematics\\
	University of Minnesota\\ 
	Minneapolis, MN 55455, USA }

\email{lyan@umn.edu}

\maketitle

\begin{abstract}
We consider an inverse problem for the nonlinear Boltzmann equation with a time-dependent kernel in dimensions $n\ge 2$. We establish a logarithm-type stability result for the collision kernel from measurements under certain additional conditions. A uniqueness result is derived as an immediate consequence of the stability result. 
Our approach relies on second-order linearization, multivariate finite differences, as well as the stability of the light-ray transform.
\end{abstract}
 
\section{Introduction}
This article studies an inverse problem for the Boltzmann equation, which describes the evolution of a dilute gas with binary collisions between particles. 
Let $\Omega\subset \R^n$, $n\ge 2$, be an open bounded and convex domain with smooth boundary $\p \Omega$. We denote 
$$
    U := \Omega\times\R^n,\quad U_T:=(0,T)\times U,\quad \Gamma^T_\pm:=(0,T)\times\Gamma_\pm,
$$
where $T>0$. The incoming ($\Gamma_{-}$) and outgoing ($\Gamma_{+}$) sets are defined as follows:
\begin{equation}
\Gamma_{\pm} := \{(x, v) \in \p\Omega\times \R^n: \pm v\cdot n(x) >0 \} 
\end{equation}
with the unit outer normal vector $n(x)$ at $x\in \p\Omega$. Suppose that $F\equiv F(t,x,v)$ is the distribution function for the particles at time $t\geq 0$ and position $x \in \Omega$ with velocity $v\in \R^n$. It satisfies the following initial boundary value problem for the Boltzmann equation:
\begin{equation}\label{eq_1_IBVP}
\left\{\begin{array}{rlll}
\p_t F + v\cdot \nabla_x F &=& Q(F,F)& \hbox{ in } U_T,\\
F &=& g&  \hbox{ on } \Gamma^T_- ,\\
F &=& h& \hbox{ on }\{0\}\times U.\\
\end{array}\right.
\end{equation}
The collision operator $Q$ describes particles' binary interactions and takes the form
\begin{equation}\label{DEF:Q}
Q(F_1,F_2)(t,x,v) =  \int_{\R^n}\int_{\S^{n-1}}{K(t,x,v,u,w)}[F_1(t,x,u')F_2(t,x,v')-F_1(t,x,u)F_2(t,x,v)]\,d\omega du,
\end{equation}
where $u$ and $v$ are velocities before a collision of particles with post-collision velocities  
\begin{equation}\label{incoming and outgoing velocities}
u' = u-[(u-v)\cdot \omega]\omega\quad\hbox{ and }\quad v' = v+[(u-v)\cdot \omega]\omega
\end{equation}
at an angle $\omega\in \mathbb{S}^{n-1}$. It is clear that they satisfy $u'\cdot \omega=v\cdot \omega$ and $v'\cdot \omega=u\cdot \omega$. Here the function $K$ is the collision kernel.

We say $K$ is in the \textit{admissible set} $\mathcal{M}$ if 
    $K\in L^{\infty}(U_T;L^1(\R^n\times\mathbb{S}^{n-1}))$
    and satisfies 
\begin{equation}\label{assumption} 
\begin{aligned}
    \|K\|_{L^{\infty}(U_T;L^1(\R^n\times\mathbb{S}^{n-1}))}:=\left\| \int_{\R^n}\int_{\mathbb{S}^{n-1}} |K(t,x,v,u,\omega)|\,d\omega du\right\|_{L^\infty(U_T)} <M 
\end{aligned}
\end{equation} 
for some constant $M>0$.
As shown in Theorem \ref{thm_wellposed}, these conditions guarantee that the problem is well-posed with small data $(g,h)$. Specifically, there exist $\kappa>0$ and a constant $C>0$ such that when 
\begin{equation}
(g,h)\in \mathcal{X}_\kappa 
:= \{(g,h)\in L^\infty(\Gamma_-^T)\times L^\infty(U): \|g\|_{L^\infty(\Gamma^T_-)} + \|h\|_{L^\infty(U)} \le \kappa\},
\end{equation}
the initial boundary value problem \eqref{eq_1_IBVP} has a unique solution $F \in L^\infty(U_T)$ satisfying $\|F\|_{L^\infty(U_T)}\le C\kappa$.
Hence we can now define the measurement operator $\mathcal{A}_K:\mathcal{X}_\kappa \rightarrow L^\infty(\Gamma^T_+)\times L^\infty(U)$, mapping from the incoming and initial data to the outgoing and final data, as follows:
\begin{equation}\label{DEF:A}
    \mathcal{A}_K:=(\mathcal{A}_K^{out},\mathcal{A}^{T}_K): (g,h)\mapsto  (F|_{\Gamma^T_+},F(T,\cdot,\cdot) )
\end{equation}
with the norm 
$$
    \|\mathcal{A}_{K}(g,h) \|_{L^\infty(\Gamma^T_+)\times L^\infty(U)} := \| \mathcal{A}^{out}_{K}(g,h)\|_{L^\infty(\Gamma^T_+)}+\| \mathcal{A}^T_{K}(g,h)\|_{L^\infty(U)}.
$$

The inverse problem here is to recover properties of the kernel $K$ from $\mathcal{A}_K$, which collected data from the surface of the domain, initial data, and final data. To understand the underlying essential details of entangled particle interactions, described by the collision operators, has been attractive not only because of theoretical interests but also its potential applications \cite{LL2000}. 

In recent years, relevant studies on identifying coefficients in the Boltzmann equation and transport equation have made substantial progress in both analytical and numerical aspects. They are motivated by a broad domain of fields, such as medical imaging \cite{Arridge1999}, astronomy \cite{WandUeno}, and remote sensing \cite{Liou2002, MarshakDavis}, see also the review articles \cite{bal_inverse_2009, Kuireview, Stefanov2003}. 
In particular, analytic methods for the determination of both absorption and scattering coefficients in the radiative transfer equation (RTE), known as the linear Boltzmann equation, are well-established. 
One crucial strategy, established in \cite{CS1, CS3} for the uniqueness result, is to analyze the singular decomposition of the kernel of the Albedo operator, which maps from incoming data to outgoing data. As a result, the absorption and scattering coefficients can be reconstructed separately because of their appearance in the components with different degrees of singularities.  
The demonstration of the method was discussed in \cite{CS1, CS2, CS3, CS98, MCDOWALL04, SU2d} for the uniqueness result and in \cite{Bal14, Bal10, Bal18, lai_inverse_2019, Wang1999, ZhaoZ18} for the corresponding stability estimate.
The application of Carleman estimates is another basic strategy to recover unknown coefficients by using suitable Carleman weight in an $L^2$ weighted estimate for a solution to a transport equation with large parameters \cite{Gaitan14,  Yamamoto2016, KlibanovP2006, Klibanov08, Lai-L, LaiUhlmannZhou22, Machida14}. The highlight of this approach is that a single measurement is sufficient to recover the linear coefficient in many cases.
And lastly, in dealing with time-dependent coefficients in the linear Boltzmann equation, \cite{bellassouedInverseProblemLinear2019} gave the uniqueness and logarithmic stability results for the time-dependent absorption coefficient and \cite{BELLASSOUED19} considered the uniqueness for the time-dependent scattering coefficient.

The literature on inverse problems for nonlinear Boltzmann equations is relatively sparse compared to the linear ones. One key factor that contributes to recovery difficulty is the special feature of the collision kernel in \eqref{DEF:Q}, which highly depends on the velocities before and after a collision. Recently, the higher-order linearization technique has been utilized to deal with inverse problems for various types of nonlinear PDEs. See for instance, \cite{KLU2018, LUW2018, UW2020} for nonlinear hyperbolic equations and \cite{AZ2018, AZ2020, feizmohammadi_inverse_2020,krupchyk_partial_2020,krupchyk_remark_2020,lai_partial_2020,lassas_partial_2020, LLLS201903} for nonlinear elliptic equations. 
The technique was generalized to the stationary nonlinear Boltzmann equation in \cite{laiReconstructionCollisionKernel2021}, in which the collision kernel was uniquely recovered under a monotonicity condition and a reconstruction formula was given. 
For the dynamic nonlinear Boltzmann equation, the uniqueness result was achieved in \cite{li_determining_2022} by performing linearization near the equilibrium. See also \cite{balehowsky_inverse_2022} for the nonlinear Boltzmann equation on Lorentzian space-time with unknown metric, where the unknown metric was determined up to an isometry from a source-to-solution map. 
Additionally, the RTE with nonlinear terms was investigated in \cite{LaiRenZhou2022,LaiUhlmannZhou22}.

\subsection{Main result}
The goal of the present work is to give uniqueness and stability results for the nonlinear Boltzmann equation with a time-dependent collision kernel with given $\mathcal{A}_K$. 
Specifically, we will study the case that $K$ has a product structure, i.e., 
$$
    K(t,x,v,u,\omega) = \Phi(t,x,|v|)\Psi(v,u,\omega).
$$
Here the unknown function $\Phi$ depending on $(t,x,v)\in \R\times\R^n\times\R^n$ is a radial function in $v$, which depends only on the speed, not the direction. Let us denote $r = |v|>0$ and write $\Phi = \Phi(t,x,r)$. The given nonnegative function $\Psi$ satisfies
\begin{align}\label{CON:Psi}
    \Psi(v,\cdot,\cdot)\geq c_0>0 \quad \hbox{ in }B_3(v)\times \mathbb{S}^{n-1}
\end{align}
for some constant $c_0>0$ independent of $v$. Note that this lower bound condition will be crucial to control the light ray transform of $\Phi$, see Section~\ref{sec:reconstruction of kernel} for details. The notation $B_\alpha(x)$ is used to represent an open ball of radius $\alpha>0$ and center $x$ in both $\R^n$ and $\R^{n+1}$ if no ambiguities arise. When $x=0$, we simply denote $B_\alpha\equiv B_\alpha(0)$.

\begin{thm}\label{thm:stability}
    Let $\Omega\subseteq\R^n$, $n\ge 2$, be an open bounded and convex domain with smooth boundary $\p \Omega$. Let $\Psi\in C^{1}(\R^n\times\R^n\times\mathbb{S}^{n-1})$ be a nonnegative function satisfying \eqref{CON:Psi}, and let $\Phi_\ell \in C^{1}((0,T)\times\Omega\times(0,\infty))$ for $\ell=1,\,2$.
    Suppose for fixed $r = |v|>0$, we have
$$
    \|\Phi_\ell(\cdot,\cdot,r)\|_{C((0,T)\times\Omega)}\leq M_1\quad \hbox{ and }\quad\text{supp}(\Phi_1-\Phi_2)(\cdot,\cdot,r)\subset (0,T)\times\Omega
$$ 
for a fixed constant $M_1>0$.

Let $\mathcal{A}_{K_\ell}$ be the measurement operator of the problem \eqref{eq_1_IBVP} with the collision kernel 
    $$ 
        K_\ell(t,x,v,u,\omega) = \Phi_\ell(t,x,r)\Psi(v,u,\omega),  
    $$
    satisfying $K_\ell,\, \nabla_x K_\ell,\,\nabla_v K_\ell\in \mathcal{M}$. Suppose that there exist constants $\kappa>0$ and $0<\delta<1$ 
    so that   
\begin{equation}\label{eq_stability_assumption}
    \|(\mathcal{A}_{K_1}- \mathcal{A}_{K_2})(g,h)\|_{L^\infty(\Gamma^T_+)\times L^\infty(U)}
\le \delta\quad \hbox{for all $(g,h)\in \mathcal{X}_\kappa$}.
\end{equation}
Then 
there exists $\mu\in(0,1)$ depending on $n,\,T$, and $\Omega$ such that 
\[
\|(\Phi_1-\Phi_2)(\cdot,\cdot,r)\|_{H^{-1}((0,T)\times\Omega)}
\le 
C \LC \delta^{\mu\over 2(n+3)} + |\log\delta|^{-1}\RC.
\]
Here $C>0$ depends only on r, $n$, $\kappa$,  $c_0$, $\Omega$, $T$, $M$, and $M_1$.

Moreover, let $s>[\frac{n+1}{2}]+1$. Suppose that 
$\Phi_\ell\in  H^s((0,T)\times \Omega\times\R)$
for $\ell=1,2$, satisfies
$$
\|\Phi_\ell(\cdot,\cdot,r)\|_{H^s((0,T)\times \Omega)}\le M_2\quad \hbox{ and }\quad\text{supp}(\Phi_1-\Phi_2)(\cdot,\cdot,r)\subset (0,T)\times\Omega
$$
for some fixed constant $M_2>0$.
Then there exists $\theta\in (0,1)$ depending on $s$ such that 
\[
\|(\Phi_1-\Phi_2)(\cdot,\cdot,r)\|_{L^\infty((0,T)\times\Omega)}
\le 
C \LC \delta^{\mu\over 2(n+3)} + |\log\delta|^{-1}\RC^\theta,
\]
where $C>0$ depends on $r$, $n$, $\kappa$,  $c_0$, $s$, $\Omega$, $T$, $M$, and $M_j$ for $j=1, 2$.
\end{thm}

The proof of Theorem~\ref{thm:stability} is given in Section~\ref{sec:reconstruction of kernel}. In order to fully recover the unknown $\Phi$, we rely on boundary measurements, initial data, and final data.  

As an immediate consequence of Theorem \ref{thm:stability}, we have the following uniqueness result.
\begin{thm}\label{thm:unique} 
Under the same hypotheses in Theorem \ref{thm:stability}, instead of \eqref{eq_stability_assumption}, suppose that we have
\[
\mathcal{A}_{K_1}(g,h) = \mathcal{A}_{K_2}(g,h) \quad \hbox{for all $(g,h)\in \mathcal{X}_\kappa$.}
\]
Then $\Phi_1 =\Phi_2 $ in $(0,T)\times\Omega\times (0,\infty)$. 
\end{thm}

All the aforementioned results for the Boltzmann equations and transport equations, except \cite{bellassouedInverseProblemLinear2019, BELLASSOUED19}, are concerned with time-independent coefficients. To the best of our knowledge, the developed methodologies, including singular decomposition method and the Carleman estimate, have been applied only for this setting. In particular, when applying the Carleman estimate, the unknown coefficients are encoded in the initial data of the solution so that it is not applicable to deal with the time-dependent coefficients, see for instance \cite{Yamamoto2016, Machida14} for the detailed discussions of the method.

In the present work, inspired by \cite{bellassouedInverseProblemLinear2019}, we rely on several techniques to determine the time-dependent kernel in \eqref{eq_1_IBVP}. To begin with, we perform the higher-order linearization method by applying the multivariate finite differences method, introduced in \cite{lassasUniquenessReconstructionStability2022} for semi-linear wave equations. Through this, one can decompose the solution of the nonlinear Boltzmann equation into three components, including the solution to the linear equation, the solution to the nonhomogeneous linear equation, and the remainder higher-order term. Due to
quadratic-like nonlinearity in the collision operator, it is sufficient to introduce two small parameters in the data in this linearization process. Next, by suitably choosing linear solutions, for each fixed $v$, the scaled light-ray transform of $\Phi(t,x,|v|)$, defined by
\[
L_r \Phi(x,v,r): = \int_{\R}\ \Phi(s,x+sv,r)ds \quad \hbox{ for } (x,v)\in \R^n\times r\mathbb{S}^{n-1},
\]
can be extracted from the integral identity. This further leads to the stability estimate of the partial Fourier transform of $\Phi$ (denoted by $\widehat{\Phi}(\tau,\xi,|v|)$)
with respect to $t$ and $x$ variables in the spacelike cone. More general results on the injectivity, and support theorems of the light-ray transform can be found in \cite{StefanovScattering89,StefanovSupportTheorems}.
The main result in Theorem~\ref{thm:stability} then follows by generalizing the stability to the timelike cone by noting the analyticity of $\widehat\Phi$ if $\Phi(\cdot,\cdot,|v|)$ is compactly supported, see Section~\ref{sec:reconstruction of kernel}. 
The uniqueness result in Theorem~\ref{thm:unique} can be viewed as a complement of the earlier work \cite{laiReconstructionCollisionKernel2021}.

The structure of the paper is as follows. In Section 2, we introduce preliminary results for a linear transport equation and establish the well-posedness of the Boltzmann equation \eqref{eq_1_IBVP} for small data. We derive an integral identity in Section 3 that bridges the unknown coefficient and the given data by using multivariate finite differences. Section 4 is devoted to the proof of Theorem \ref{thm:stability} and Theorem \ref{thm:unique}. 

\section{Preliminaries}
In this section, we introduce preliminary results for a linear transport equation and establish the well-posedness of the Boltzmann equation \eqref{eq_1_IBVP} for small data.

\subsection{Notations}

We define the forward exit time $\tau_+(x,v)$ and the backward exit time $\tau_-(x,v)$ for every $(x,v)\in \overline\Omega\times \R^n$ by
\[
    \tau_{\pm}(x,v) := \sup\{s\ge 0: x\pm sv \in \Omega\} .
\]
In other words, $\tau_{+}(x,v)$ and $\tau_{-}(x,v)$ are the time at which a particle $x$ leaves the domain $\Omega$ with velocity $v$ and velocity $-v$, respectively.

We define the space $L^p(\Gamma^T_\pm;d\xi)$, $1\leq p<\infty$, with the norm
$$
   \|f\|_{L^p(\Gamma^T_\pm;d\xi)}= \LC\int^T_0\int_{\Gamma_\pm} |f|^p\,d\xi dt\RC^{1/p},
$$
where $d\xi=|n(x)\cdot v|d\sigma_x dv$ is the measure on $\Gamma_\pm$ with the measure $d\sigma_x$ on $\p\Omega$.
For $p=\infty$, we denote $L^\infty(U)$, $L^\infty(U_T)$, and $L^\infty(\Gamma_\pm^T)$ to be the standard spaces consisting of all functions that are essentially bounded.

\subsection{Forward problems}
We first study the existence and stability of solutions to the initial boundary value problem for a linear transport equation. The following lemma can be found in \cite[Proposition 4, Chapter 21]{dautrayMathematicalAnalysisNumerical2000}, in which more discussions on the well-posed problems for transport equations are addressed. In the general geometry setting, we refer to \cite{LaiUhlmannZhou22} for details.
\begin{lemma}\label{lemma_solution}
Suppose that $f\in L^\infty(U_T)$, $g\in L^\infty(\Gamma^T_-)$ and $h\in L^\infty(U)$. Then the problem  
\begin{equation}\label{linear equation}
\left\{\begin{array}{rlll}
\p_t F + v\cdot \nabla_x F &=& f& \hbox{ in } U_T,\\
F &=& g&  \hbox{ on } \Gamma^T_- ,\\
F &=& h& \hbox{ on }\{0\}\times U ,\\
\end{array}\right.
\end{equation}
has a unique solution $F$ in $L^\infty(U_T)$ satisfying
\begin{align}\label{eq_2_solution}
F(t,x,v) &= g(t-\tau_{-}(x,v),x-\tau_{-}(x,v)v,v)H(t-\tau_{-}(x,v)) + h(x-tv,v)H(\tau_{-}(x,v)-t) \notag\\
&\quad + \int_{0}^{t}f(t-s, x-sv,v)H(\tau_{-}(x,v)-s)\,ds,
\end{align}
where $H$ is the Heaviside function satisfies $H(s)=0$ if $s<0$ and $H(s)=1$ if $s>0$. 
Moreover, $F$ satisfies
\begin{equation}\label{EST:linear F}
\|F\|_{L^\infty(U_T)}\le \|g\|_{L^\infty(\Gamma^T_-)} + \|h\|_{L^\infty(U)} + T \|f\|_{L^\infty(U_T)}. 
\end{equation}
\end{lemma}

With the above result, we then show that the initial boundary value problem for \eqref{eq_1_IBVP} is well-posed for small data by applying the contraction mapping principle. 

\begin{thm}[Well-posedness of the Boltzmann equation] 
\label{thm_wellposed}
Let $\Omega\subseteq\R^n$, $n\ge 2$, be an open bounded and convex domain with smooth boundary $\p \Omega$. Suppose $K\in \mathcal{M}$. Then there exist constants $\kappa>0$ and $C>0$ such that when $(g,h)\in 
\mathcal{X}_\kappa$, the problem \eqref{eq_1_IBVP} has a unique solution $F\in L^\infty(U_T)$ satisfying
\[
\|F\|_{L^\infty(U_T)}\le C\LC\|g\|_{L^\infty(\Gamma^T_-)} + \|h\|_{L^\infty(U)}\RC,
\]
where the constant $C$ depends on $\kappa$, $M$, and $T$.
\end{thm}

\begin{proof}
Given $(g,h)\in \mathcal{X}_\kappa$, from Lemma \ref{lemma_solution}, there exists a solution $\widetilde F\in L^\infty(U_T)$ to the following homogeneous equation:
\begin{align*}
\left\{\begin{array}{rlll}
\p_t \widetilde F + v\cdot \nabla_x \widetilde F &=& 0& \hbox{ in } U_T,\\
\widetilde F &=& g&  \hbox{ on } \Gamma^T_- ,\\
\widetilde F &=& h& \hbox{ on }\{0\}\times U ,
\end{array}\right.
\end{align*}
satisfying 
\begin{equation}\label{eq_2_lemma_est}
\|\widetilde F\|_{L^\infty(U_T)} \le \|g\|_{L^\infty(\Gamma^T_-)} + \|h\|_{L^\infty(U)} \le \kappa.
\end{equation}

To show the existence of solution $F$ of \eqref{eq_1_IBVP}, we observe that if we set $G := F-\widetilde F$, then $G$ will satisfy
\begin{align*}
\left\{\begin{array}{rlll}
\p_t G+ v\cdot \nabla_x G &=& Q(\widetilde F+ G, \widetilde F+ G)=: Q_0(G)& \hbox{ in } U_T,\\
G &=& 0&  \hbox{ on } \Gamma^T_- ,\\
G &=& 0& \hbox{ on }\{0\}\times U.
\end{array}\right.
\end{align*}
Equivalently, $G$ solves $G=(\mathcal{L}^{-1}\circ Q_0)(G)$,
where $\mathcal{L}^{-1}$ denotes the solution operator of \eqref{linear equation} with $g=h=0$.
To find $G$, we will apply the contraction mapping principle and show $\mathcal{L}^{-1}\circ Q_0$ is a contraction map on a subset $\mathfrak{X}$ of $L^\infty(U_T)$, which is defined as follows:
\begin{equation}
\mathfrak{X} := \{\varphi\in L^\infty(U_T):\, \varphi|_{\Gamma^T_-} = 0, \quad \varphi|_{t=0} = 0,\quad \hbox{and } \|\varphi\|_{L^\infty(U_T)} \le c \}
\end{equation}
with some $c>0$ to be determined later. 

For each $\varphi\in \mathfrak{X}$, by using the fact that $K\in \mathcal{M}$ and \eqref{eq_2_lemma_est}, Lemma~\ref{lemma_solution} yields that 
\begin{align*}
    \|(\mathcal{L}^{-1}\circ Q_0)(\varphi)\|_{L^\infty(U_T)} 
    \leq 
    T\|Q_0(\varphi)\|_{L^\infty(U_T)} 
    \leq 2M T \LC \|\widetilde F\|_{L^\infty(U_T)} +\|\varphi\|_{L^\infty(U_T)}\RC^2 
    \leq 2MT\LC\kappa + c \RC^2.
\end{align*}
Moreover, for $\varphi_1,\,\varphi_2\in \mathfrak{X}$, we can derive that
\begin{align*}
    \|(\mathcal{L}^{-1}\circ Q_0)(\varphi_1)-(\mathcal{L}^{-1}\circ Q_0)(\varphi_1)\|_{L^\infty(U_T)} 
    &\leq T\|Q_0(\varphi_1)- Q_0(\varphi_2)\|_{L^\infty(U_T)} 
    \\
    &\leq 4MT (\kappa+c)\|\varphi_1-\varphi_2\|_{L^\infty(U_T)}.
\end{align*}
We take $\kappa$ and $c$ sufficiently small so that $0<\kappa<c<1$ and 
$$
2MT\LC\kappa + c\RC^2 \leq c \quad\hbox{ and }\quad 4MT (\kappa+c)<1.
$$
This implies that $\mathcal{L}^{-1}\circ Q_0$ is a contraction map. By the contraction mapping principle, there exists a unique fixed point $G\in\mathfrak{X}$ so that $G$ satisfies $G =(\mathcal{L}^{-1}\circ Q_0)(G) $ and
\begin{align*}
\|G\|_{L^\infty(U_T)} 
&= \|(\mathcal{L}^{-1}\circ Q_0)(G)\|_{L^\infty(U_T)}  \\
&\le 2MT \LC\|\widetilde F\|_{L^\infty(U_T)}+ \|G\|_{L^\infty(U_T)}\RC^2 \\
&\le 2MT(\kappa+c)\LC\|\widetilde F\|_{L^\infty(U_T)} + \|G\|_{L^\infty(U_T)}\RC.
\end{align*}
Due to the choice of $c$ and $\kappa$, we have $2MT(\kappa+c)<1$, which implies that the second term on the right-hand side above can be absorbed by the left. Thus,
\[
\|G\|_{L^\infty(U_T)}\le C\|\widetilde F\|_{L^\infty(U_T)}
\]
for some constant $C>0$. 
We now conclude that $F = \widetilde F+G$ is a solution of \eqref{eq_1_IBVP} and satisfies  
\[
\|F\|_{L^\infty(U_T)}  \le  \|\widetilde F\|_{L^\infty(U_T)} +\|G\|_{L^\infty(U_T)}\le C \LC\|g\|_{L^\infty(\Gamma_-^T)} + \|h\|_{L^\infty(U)}\RC 
\]
by noting that \eqref{eq_2_lemma_est}. This completes the proof.
\end{proof}

\section{Integral Identity}\label{section_integral_identity}
The section aims to derive an integral identity for the unknown coefficient. To do so, we study the expansion formula for a family of solutions depending on small parameters $\varepsilon = (\varepsilon_1, \varepsilon_2)$ by applying multivariate finite differences.

\subsection{Multivariate finite difference}
Let $f\equiv f(\varepsilon_1,\varepsilon_2)$ represent a function of two variables $\varepsilon:=(\varepsilon_1,\varepsilon_2),  \varepsilon_j\ge 0, j = 1,2$. 
We define the \textit{second-order finite difference operator} $D^2_{\varepsilon_1, \varepsilon_2}$ of the function $f$ at the point $\varepsilon=0$ for $\varepsilon_j >0, j = 1,2$ by
$$
D^2_{\varepsilon_1, \varepsilon_2} |_{\varepsilon=0}f := {1\over \varepsilon_1\varepsilon_2}\LC f(\varepsilon_1,\varepsilon_2) - f(\varepsilon_1,0)-f(0,\varepsilon_2)- f(0,0)\RC.
$$

By Theorem \ref{thm_wellposed}, for given data $g_j\in L^\infty(\Gamma^T_-)$ and $h_j\in L^\infty(U)$, $j=1,2$ with sufficiently small parameters $\varepsilon_1\ge 0$ and $\varepsilon_2\ge 0$ such that
$\sum_{j=1}^2 \varepsilon_j\LC  \|g_j\|_{L^\infty(\Gamma^T_-)}+ \|h_j\|_{L^\infty(U)}\RC\leq \kappa$, 
there exists a unique solution $F_{\varepsilon g}^{\varepsilon h}\equiv F_{\varepsilon_1 g_1+\varepsilon_2g_2}^{\varepsilon_1 h_1+\varepsilon_2h_2}\in L^\infty(U_T)$ to the initial boundary value problem: 
\begin{equation}\label{eq_3_lineariztion_F}
\left\{\begin{array}{rlll}
    \p_t F_{\varepsilon g}^{\varepsilon h}+v\cdot\nabla_x F_{\varepsilon g}^{\varepsilon h}&=& Q(F_{\varepsilon g}^{\varepsilon h},F_{\varepsilon g}^{\varepsilon h})& \hbox{ in } U_T,\\
    F_{\varepsilon g}^{\varepsilon h} &=& \varepsilon_1 g_1+\varepsilon_2g_2 &  \hbox{ on } \Gamma^T_- ,\\
    F_{\varepsilon g}^{\varepsilon h} &=& \varepsilon_1 h_1+\varepsilon_2h_2& \hbox{ on }\{0\}\times U .\\
\end{array}\right.
\end{equation}
Moreover, the solution satisfies the estimate 
\begin{equation}\label{eq_3_estimate_F}
    \|F_{\varepsilon g}^{\varepsilon h}\|_{L^\infty(U_T)} \le  C \left(\|\varepsilon_1 g_1+\varepsilon_2g_2\|_{L^\infty(\Gamma^T_-)} + \|\varepsilon_1 h_1+\varepsilon_2h_2\|_{L^\infty(U)}\right).
\end{equation}
In particular, by the definition of the second-order finite difference, we have
\begin{equation}\label{eq_3_finite_difference}
    D^2_{\varepsilon_1,\varepsilon_2}|_{\varepsilon=0}F_{\varepsilon g}^{\varepsilon h}
    = \frac{1}{\varepsilon_1\varepsilon_2} \LC F_{\varepsilon g}^{\varepsilon h} - F_{\varepsilon_1g_1}^{\varepsilon_1h_1} -F_{\varepsilon_2g_2}^{\varepsilon_2h_2}\RC 
\end{equation}
due to the fact that the solution $F_{\varepsilon g}^{\varepsilon h}$ is identically $0$ when $\varepsilon_1 = \varepsilon_2 = 0$.

In order to express $F_{\varepsilon g}^{\varepsilon h}$ in terms of $\varepsilon_j$, we define two solutions $V_j$ and $W_{(k_1,k_2)}$ as follows. First, let $V_j$, $j = 1,2$ be the unique solution to the problem:
\begin{equation}\label{eq_3_linearization_V}
    \left\{\begin{array}{rlll}
    \p_t V_j + v\cdot \nabla_x V_j &=& 0& \hbox{ in } U_T,\\
    V_j &=& g_j&  \hbox{ on } \Gamma^T_- ,\\
    V_j &=& h_j& \hbox{ on }\{0\}\times U ,
    \end{array}\right.
\end{equation}
and by Lemma \ref{lemma_solution}, it satisfies that 
\begin{equation}
    \|V_j\|_{L^\infty(U_T)}\le \|g_j\|_{L^\infty(\Gamma_-^T)}+\|h_j\|_{L^\infty(U)}.
\end{equation}

In addition, for $k_1,k_2\in\{0,1,2\}$ satisfying $k_1+k_2 = 2$, Lemma \ref{lemma_solution} shows that there exist solutions $W_{(k_1,k_2)}\in L^\infty(U_T)$ to the nonhomogeneous equation: 
\begin{equation}\label{eq_3_linearization_W}
    \left\{\begin{array}{rlll}
    \p_t W_{(k_1,k_2)} + v\cdot \nabla_x W_{(k_1,k_2)} &=& S_{(k_1,k_2)}& \hbox{ in } U_T,\\
    W_{(k_1,k_2)} &=& 0&  \hbox{ on } \Gamma^T_- ,\\
    W_{(k_1,k_2)} &=&0& \hbox{ on }\{0\}\times U ,
\end{array}\right.
\end{equation}
where the source terms $S_{(k_1,k_2)}$ are defined as follows:
\begin{align*}
    S_{(2,0)}(t,x,v) := \int_{\R^n}\int_{\mathbb{S}^{n-1}}K(t, x, v,u,\omega)[&V_{1}(t,x,u')V_{1}(t,x,v')
    - V_{1}(t,x,u)V_{1}(t,x,v)\,d\omega du,
\end{align*}
\begin{align}\label{DEF:S11}
    S_{(1,1)}(t,x,v) := \int_{\R^n}\int_{\mathbb{S}^{n-1}}K(t, x, v,u,\omega)[&V_{1}(t,x,u')V_{2}(t,x,v')+V_{1}(t,x,v')V_{2}(t,x,u') \notag\\
    &- V_{1}(t,x,u)V_{2}(t,x,v)-V_{1}(t,x,v)V_{2}(t,x,u)]\,d\omega du,
\end{align}
and  
\begin{align*}
S_{(0,2)}(t,x,v) := \int_{\R^n}\int_{\mathbb{S}^{n-1}}K(t, x, v,u,\omega)[&V_{2}(t,x,u')V_{2}(t,x,v')
- V_{2}(t,x,u)V_{2}(t,x,v)\,d\omega du.
\end{align*}

\begin{prop}\label{prop_3_1}
Suppose that there exists $\kappa>0$ such that if $g_j\in L^\infty(\Gamma^T_-)$ and $h_j\in L^\infty(U)$ and $\varepsilon_j\ge 0$, $j=1,2$, satisfy $\sum_{j=1}^2 \varepsilon_j\LC  \|g_j\|_{L^\infty(\Gamma^T_-)}+ \|h_j\|_{L^\infty(U)}\RC\leq \kappa$, then there exists a unique solution $F_{\varepsilon g}^{\varepsilon h}$ to \eqref{eq_3_lineariztion_F} satisfying the estimate \eqref{eq_3_estimate_F}. Moreover, $F_{\varepsilon g}^{\varepsilon h}$ can be expanded in $\varepsilon_1$ and $\varepsilon_2$ in the following form:
\begin{equation}\label{eq_3_expansion}
    F_{\varepsilon g}^{\varepsilon h} = \varepsilon_1V_1+\varepsilon_2V_2+ \varepsilon_1^2W_{(2,0)}+ \varepsilon_{1}\varepsilon_{2} W_{(1,1)} + \varepsilon_2^2W_{(0,2)} +\mathcal{R}_\varepsilon,
\end{equation}
where $V_j$, $j=1,2$, solve
\eqref{eq_3_linearization_V} and $W_{(k_1,k_2)}$, $k_1,k_2\in \{0,1,2\},\, k_1+k_2 = 2$ solve \eqref{eq_3_linearization_W}. 

Moreover, the remaining function $\mathcal{R}_\varepsilon$ solves the nonhomogeneous equation:
\begin{equation}\label{eq_3_linearization_R}
    \left\{\begin{array}{rlll}
    \p_t\mathcal{R}_\varepsilon +v\cdot \nabla_x\mathcal{R}_\varepsilon &=& Q(F_{\varepsilon g}^{\varepsilon h},F_{\varepsilon g}^{\varepsilon h})  -\varepsilon_1^2S_{(2,0)}- \varepsilon_{1}\varepsilon_{2} S_{(1,1)} - \varepsilon_2^2S_{(0,2)}
    & \hbox{ in } U_T,\\
    \mathcal{R}_\varepsilon &=&0 &\hbox{ on } \Gamma^T_- ,\\
    \mathcal{R}_\varepsilon &=& 0& \hbox{ on }\{0\}\times U ,\\
    \end{array}\right.
\end{equation}
and satisfies the following estimate 
\begin{equation}\label{eq_3_estimate_R}
    \|\mathcal{R}_\varepsilon\|_{L^\infty(U_T)} 
    \le  C\LC \|\varepsilon_1g_1+\varepsilon_2g_2\|_{L^\infty(\Gamma^T_-)}+ \|\varepsilon_1h_1+\varepsilon_2h_2\|_{L^\infty(U)}\RC^3,
\end{equation}
where $C$ depends on $\kappa$, $M$, and $T$.
\end{prop}

\begin{proof}  
It remains to show the existence of the solution $\mathcal{R}_\varepsilon$ and \eqref{eq_3_estimate_R}. To this end, we first denote the function $\mathcal{F}_\varepsilon$ by
$$
    \mathcal{F}_\varepsilon:=F_{\varepsilon g}^{\varepsilon h}-\varepsilon_1V_1-\varepsilon_2V_2.
$$
Then $\mathcal{F}_\varepsilon\in L^\infty(U_T)$ solves $\p_t\mathcal{F}_\varepsilon+v\cdot\nabla_x \mathcal{F_\varepsilon} = Q(F_{\varepsilon g}^{\varepsilon h},F_{\varepsilon g}^{\varepsilon h})$ with zero initial and boundary data. Applying Lemma~\ref{lemma_solution} and $K\in\mathcal{M}$, we have
\begin{align}\label{EST:P}
    \|\mathcal{F}_\varepsilon\|_{L^\infty(U_T)}\leq T\|Q(F_{\varepsilon g}^{\varepsilon h},F_{\varepsilon g}^{\varepsilon h})\|_{L^\infty(U_T)}\leq 2MT \|F_{\varepsilon g}^{\varepsilon h}\|_{L^\infty(U_T)}^2.
\end{align}
Moreover, we denote the nonhomogeneous term in \eqref{eq_3_linearization_R} by
$$
    \mathfrak S_\varepsilon(t,x,v):=Q(F_{\varepsilon g}^{\varepsilon h},F_{\varepsilon g}^{\varepsilon h})  -\varepsilon_1^2S_{(2,0)}- \varepsilon_{1}\varepsilon_{2} S_{(1,1)} - \varepsilon_2^2S_{(0,2)}.
$$
With \eqref{eq_3_estimate_F}, \eqref{EST:P} and $K\in \mathcal{M}$, for $(t,x,v)\in U_T$, a direct computation gives
\begin{equation}\label{eq_3_estimate_R_1}
\begin{aligned}
    |\mathfrak S_\varepsilon(t,x,v)|
    &= \bigg|\int_{\R^n}\int_{\S^{n-1}}K\Big[ \mathcal{F}_\varepsilon(t,x,u') F_{\varepsilon g}^{\varepsilon h}(t,x,v') + (\varepsilon_1V_1+\varepsilon_2V_2)(t,x,u')\mathcal{F}_\varepsilon(t,x,v') \\
    &\quad - \mathcal{F}_\varepsilon(t,x,u) F_{\varepsilon g}^{\varepsilon h}(t,x,v) - (\varepsilon_1V_1+\varepsilon_2V_2)(t,x,u)\mathcal{F}_\varepsilon(t,x,v) \Big]\,d\omega du\bigg|\\
    &\leq 4M^2T \|F_{\varepsilon g}^{\varepsilon h}\|^2_{L^\infty(U_T)}\LC \|F_{\varepsilon g}^{\varepsilon h}\|_{L^\infty(U_T)} + \|\varepsilon_1V_1+\varepsilon_2V_2 \|_{L^\infty(U_T)}\RC\\
    &\le C\left(\|\varepsilon_1 g_1+\varepsilon_2g_2\|_{L^\infty(\Gamma^T_-)} + \|\varepsilon_1 h_1+\varepsilon_2h_2\|_{L^\infty(U)}\right)^3.
\end{aligned}
\end{equation}
Finally, applying Lemma \ref{lemma_solution} to $\mathcal{R}_\varepsilon$, we obtain \eqref{eq_3_estimate_R}.
\end{proof}

Using the expansion of $F_{\varepsilon g}^{\varepsilon h}$ in \eqref{eq_3_expansion}, we can rewrite the finite difference $D^2_{\varepsilon_1,\varepsilon_2}|_{\varepsilon = 0} F_{\varepsilon g}^{\varepsilon h}$, defined in \eqref{eq_3_finite_difference}, as follows:
\begin{equation}\label{ID:finte diff}
    D^2_{\varepsilon_1,\varepsilon_2}|_{\varepsilon = 0} F_{\varepsilon g}^{\varepsilon h} = W_{(1,1)}+D^2_{\varepsilon_1,\varepsilon_2}|_{\varepsilon = 0}\mathcal{R}_\varepsilon.  
\end{equation}
Therefore, by applying the transport operator $\p_t+v\cdot\nabla_x$ to both sides of \eqref{ID:finte diff}, since $\mathcal{R}_\varepsilon$ is a solution of \eqref{eq_3_linearization_R}, we obtain 
\begin{equation}\label{eq_3_equality_FD}
    (\p_t+v\cdot\nabla_x) D^2_{\varepsilon_1,\varepsilon_2}|_{\varepsilon = 0}F_{\varepsilon g}^{\varepsilon h}=  S_{(1,1)}+  D^2_{\varepsilon_1,\varepsilon_2}|_{\varepsilon = 0}\mathfrak S_\varepsilon.  
\end{equation}

\subsection{An integral identity}
Equipped with the definitions and discussions above, we are ready to derive the key identity. 
Let $\varphi\equiv \varphi(x,v)$ be a function in the space $C^\infty_c (\R^n \times\R^n)$ and then take $$
    \psi(t,x,v):=\varphi(x-tv,v)
$$ 
for $(t,x,v)\in\R\times\R^n\times\R^n$. It is clear that $\psi$ is the solution to the transport equation
\[
	\p_t \psi + v\cdot \nabla_x \psi= 0 \quad \hbox{ in }\R\times\R^n\times\R^{n}  
\] 
with the initial data $\psi|_{t=0}=\varphi$.
We have the following integral identity.
\begin{prop}\label{prop:ID}
Suppose all the hypotheses in Proposition \ref{prop_3_1} are satisfied.
Then the following integral identity holds:
\begin{equation}\label{eq_3_integral_identity}
    \begin{aligned}
    \int_{U_T} \psi S_{(1,1)} \,dtdxdv 
    =&\int_U  \psi\bigg[D^2_{\varepsilon_1,\varepsilon_2}|_{\varepsilon=0}\mathcal{A}^T_K(\varepsilon_1g_1+\varepsilon_2g_2,\varepsilon_1h_1+\varepsilon_2h_2)\bigg](T,x,v)\,dxdv\\
    &+ \int_{\Gamma^T_+} (v\cdot n(x)) \psi \bigg[D^2_{\varepsilon_1,\varepsilon_2}|_{\varepsilon = 0}\mathcal{A}^{out}_K(\varepsilon_1g_1+\varepsilon_2g_2,\varepsilon_1h_1+\varepsilon_2h_2)\bigg]  \,dt d\sigma_x dv\\
    & - \int_{U_T} \psi \left[ D^2_{\varepsilon_1,\varepsilon_2}|_{\varepsilon=0}\mathfrak S_\varepsilon\right] \, dtdxdv,
    \end{aligned}
\end{equation}
where $d\sigma_x$ is the surface measure of $\p\Omega$.
\end{prop}
\begin{proof}
Based on the definition of $D^2_{\varepsilon_1,\varepsilon_2}|_{\varepsilon= 0}F_{\varepsilon g}^{\varepsilon h}$, it has trivial initial and boundary data, that is,
\begin{align*}\LC D^2_{\varepsilon_1,\varepsilon_2}|_{\varepsilon= 0}F_{\varepsilon g}^{\varepsilon h}\RC\Big|_{t=0} =\LC D^2_{\varepsilon_1,\varepsilon_2}|_{\varepsilon= 0}F_{\varepsilon g}^{\varepsilon h}\RC\Big|_{\Gamma^T_-}= 0.
\end{align*}
With this, multiplying \eqref{eq_3_equality_FD} by $\psi$ and integrating over the domain $U_T$ yield that
\begin{equation}\label{}
\begin{aligned}
\int_{U_T} \psi S_{(1,1)} \,dtdxdv 
&=\int_U  \psi \left[D^2_{\varepsilon_1,\varepsilon_2}|_{\varepsilon = 0}F_{\varepsilon g}^{\varepsilon h}\right](T,x,v) \,dxdv - \int_{U_T} \psi \left[D^2_{\varepsilon_1,\varepsilon_2}|_{\varepsilon = 0}\mathfrak{S}_\varepsilon \right]\,dtdxdv\\
&\quad + \int_{\Gamma^T_+} (v\cdot n(x)) \psi \left[D^2_{\varepsilon_1,\varepsilon_2}|_{\varepsilon = 0}F_{\varepsilon g}^{\varepsilon h}\right](t,x,v)\, dt d\sigma_x dv.
\end{aligned}
\end{equation}
This ends the proof by recalling the definitions of $\mathcal{A}^{out}_K$ and $\mathcal{A}^T_K$ in \eqref{DEF:A}.
\end{proof}

\section{Recovery of the collision kernel}\label{sec:reconstruction of kernel}
In this section, we will mainly focus on the stability estimate stated in Theorem~\ref{thm:stability} since the uniqueness result in Theorem~\ref{thm:unique} follows directly. Initiating from the integral identity in Proposition~\ref{prop:ID}, the central step is to control the light ray transform of $\Phi$ by suitably choosing functions $\varphi$ localizing at a given point $(x_*,v_*)$.

Let $F^{(\ell)}$ solve the problem \eqref{eq_3_lineariztion_F} with the collision kernel $K$ replaced by 
$$
    K_\ell(t,x,v,u,\omega) = \Phi_\ell(t,x,|v|)\Psi(v,u,\omega)
$$ in the collision operator $Q$ for $\ell=1,\,2$. We denote the functions $V^{(\ell)}_j$, $W^{(\ell)}$, $S^{(\ell)}_{(k_1,k_2)}$, $\mathcal{R}_\varepsilon^{(\ell)}$, and $\mathfrak S_\varepsilon^{(\ell)}$, $\ell=1,\,2$, to be  the corresponding functions $V_j$, $W$, $S_{(k_1,k_2)}$, $\mathcal{R}_\varepsilon$ and $\mathfrak S_\varepsilon$ in Section~\ref{section_integral_identity}. Moreover, for the purpose of simplifying the notations in Proposition~\ref{prop:ID diff} below, we denote the difference of the sources $S^{(\ell)}_{(1,1)}$ and also the difference of the unknown coefficients $\Phi_\ell$ by
$$
    \widetilde{S} :=S^{(1)}_{(1,1)}- S^{(2)}_{(1,1)}\quad \hbox{ and }\quad \widetilde\Phi := \Phi_1 - \Phi_2.
$$
Note that since $V^{(\ell)}_j$, $\ell=1,\,2$, are independent of the unknown kernel $K_\ell$ and have the same initial and boundary data, by the well-posedness result for the linear transport equation, we deduce that
\[
    V^{(1)}_j = V^{(2)}_j = : V_j, \quad j = 1,2.
\]

\begin{prop}\label{prop:ID diff} Suppose that all the hypotheses in Proposition~\ref{prop_3_1} and \eqref{eq_stability_assumption} hold. Suppose that $\Phi_\ell\in C^1((0,T)\times\Omega\times(0,\infty))$ for $\ell=1,\,2$ satisfies
$$
    \text{supp}(\Phi_1-\Phi_2)(\cdot,\cdot,|v|)\subset (0,T)\times\Omega.
$$ 
Then 
\begin{equation}\label{prop_4_1}
\begin{aligned}
&\quad \left|\int_\R \int_{\R^n} \int_{\R^{n}} \varphi(y,v)\widetilde{S}(t,y+tv,v)\,dydvdt \right| 
\\
&\leq \frac{C}{\varepsilon_1\varepsilon_2} \delta \left(\|\psi\|_{L^1(\{T\}\times U)} + \|\psi\|_{L^1(\Gamma^T_+;d\xi)} \right) 
\\
&\quad+ \frac{C}{\varepsilon_1\varepsilon_2}\|\psi\|_{L^1(U_T)} \LC\sum_{j=1}^2 \varepsilon_j \LC \|g_j\|_{L^\infty(\Gamma^T_-)}+ \|h_j\|_{L^\infty(U)}\RC \RC^3 ,
\end{aligned}
\end{equation}
where $C$ is a positive constant depending on $\kappa$, $M$, and $T$.
\end{prop}
\begin{proof}
Applying the integral identity \eqref{eq_3_integral_identity} for $K_\ell,\,\ell = 1,\,2$ and taking the difference, 
we obtain 
\begin{equation}\label{eq_4_intergral_identity_difference_1}
\begin{aligned}
    &\int_{U_T} \psi (t,x,v) \widetilde{S}(t,x,v)\, dtdxdv\\
    =\,&\int_U  \psi(T,x,v)\bigg[D^2_{\varepsilon_1,\varepsilon_2}|_{\varepsilon= 0}(\mathcal{A}^T_{K_1}-\mathcal{A}^T_{K_2})(\varepsilon_1g_1+\varepsilon_2g_2,\varepsilon_1h_1+\varepsilon_2h_2)\bigg](T,x,v) \,dxdv\\
    &+ \int_{\Gamma^T_+} (n(x)\cdot v) \psi \bigg[D^2_{\varepsilon_1,\varepsilon_2}|_{\varepsilon = 0}(\mathcal{A}^{out}_{K_1}-\mathcal{A}^{out}_{K_2})(\varepsilon_1g_1+\varepsilon_2g_2,\varepsilon_1h_1+\varepsilon_2h_2)\bigg]  \,dt d\sigma_x dv
    \\
    & - \int_{U_T} \psi\bigg[D^2_{\varepsilon_1,\varepsilon_2}|_{\varepsilon = 0}(\mathfrak S_\varepsilon^{(1)} - \mathfrak S_\varepsilon^{(2)}) \bigg] \,dtdxdv.
\end{aligned}
\end{equation}

On one hand, since $(\Phi_1-\Phi_2)(\cdot,\cdot,|v|)$ is compactly supported in $(0,T)\times\Omega$ for each fixed $v\in \R^n$, from the definition of $S^{(\ell)}_{(1,1)}$ in \eqref{DEF:S11}, we can extend $\widetilde{S}(\cdot,\cdot,v)$ to $\R\times\R^n$ by $0$ outside the domain $(0,T)\times\Omega$ while preserving the regularity. Hence, by performing the change of variable $y = x-tv$, the left-hand side of \eqref{eq_4_intergral_identity_difference_1} becomes
\[
    \int_\R \int_{\R^n} \int_{\R^{n}} \varphi(y,v)\widetilde{S}(t,y+tv,v) \,dydvdt
\]
by recalling that $\psi(t,x,v) = \varphi(x-tv,v)$. 

On the other hand, let us take $\varepsilon_1$ and $\varepsilon_2$ small enough and use \eqref{eq_stability_assumption} and \eqref{eq_3_estimate_R_1} such that the absolute value of the right-hand side of \eqref{eq_4_intergral_identity_difference_1} is bounded above by 
\begin{align*}
    \frac{C}{\varepsilon_1\varepsilon_2} & \delta \left(\|\psi\|_{L^1(\{T\}\times U)} + \|\psi\|_{L^1(\Gamma^T_+; d\xi)} \right)  
\\
    &+ \frac{C}{\varepsilon_1\varepsilon_2}\|\psi\|_{L^1(U_T)} 
    \LC\sum_{j=1}^2 \varepsilon_j \LC \|g_j\|_{L^\infty(\Gamma^T_-)}+ \|h_j\|_{L^\infty(U)}\RC \RC^3, 
\end{align*}
where $C$ depends on $\kappa$, $M$, and $T$. This ends the proof of the proposition.
\end{proof}

The \textit{light ray transform} $L$ of a function $q\in L^1(\R^{n+1})$ is defined by 
$$
    L q(x,v):=\int_\R q(s,x+sv)\,ds\quad \hbox{ for }(x,v)\in \R^n\times\mathbb{S}^{n-1}.
$$
Recall that $\widetilde\Phi$ is a radial function in $v$, for our purpose, for $r>0$, we define the \textit{scaled light ray transform} of $\widetilde\Phi$ by 
\[
L_r \widetilde\Phi(x,v,r): = \int_{\R}\widetilde\Phi(s,x+sv,r)ds \quad \hbox{ for } (x,v)\in \R^n\times r\mathbb{S}^{n-1}.
\]
\begin{prop}\label{prop_4_2}
Suppose that all the hypotheses of Theorem \ref{thm:stability} hold. 
For any $(y_*,v_*)\in \R^n\times r\S^{n-1}$, we have the following estimate: 
\begin{equation}\label{EST: light ray Phi}
    \LV L_r \widetilde \Phi(y_*,v_*,r)\RV
\le C\delta^{\frac{1}{n+3}},
\end{equation}
where $C$ depends on $r$, $n$, $\kappa$,  $c_0$, $M$, $\Omega$, and $T$. 
\end{prop}

\begin{proof}
The proof is split into three steps.\\ 
\noindent{\bf Step 1: Solutions to the linear equations.} We first choose suitable smooth functions $\varphi$ in $C^\infty_c (\R^n; C(\R^n))$ to be applied in the integral identity \eqref{prop_4_1}. To this end, let $\chi\in C_0^\infty(\R^n)$ be a smooth function satisfying $0\leq \chi\leq 1$, $\chi(0)=1$, and $\|\chi\|_{L^1(\R^n)} = 1$ with support $\hbox{supp}(\chi) \subseteq B_1$.
For a fixed $(y_*,v_*)\in \R^n\times r \S^{n-1}$, we define
\[
\varphi_\lambda(y,v) = \frac{1}{\lambda^{2n}}\chi\LC\frac{y-y_*}{\lambda}\RC \chi\LC\frac{v-v_*}{\lambda}\RC,\quad 0<\lambda<1.
\]
Then $\varphi_\lambda \in C^\infty_c(\R^n\times\R^n)$ and $\|\varphi_\lambda\|_{L^1(\R^n\times\R^n)} = 1$. Let $\psi_\lambda(t,x,v) = \varphi_\lambda(x-tv,v)$, which satisfies
\begin{align}\label{EST:phi lambda 1}
    \|\psi_\lambda\|_{L^1(\{T\}\times U)} \leq \|\varphi_\lambda\|_{L^1(\R^n\times\R^n)}=1,
\end{align}
\begin{align}\label{EST:phi lambda 3}
    \|\psi_\lambda\|_{L^1(U_T)}\leq \|\psi_\lambda\|_{L^1((0,T)\times\R^n\times\R^n)} =T\|\varphi_\lambda\|_{L^1(\R^n\times\R^n)} = T,
\end{align}
and 
\begin{align}\label{EST:phi lambda 2}
    \|\psi_\lambda\|_{L^1(\Gamma^T_+;d\xi)}&\leq \int^{T}_0\int_{\R^n}\int_{\p\Omega} \LV v\cdot n(x) \RV \varphi_\lambda(x-tv,v)\,d\sigma_x dvdt \notag\\
    &\leq (r+1)T\lambda^{-n}|\p\Omega|\int_{\R^n} \lambda^{-n}\chi \LC\frac{v-v_*}{\lambda}\RC \,dv 
    = (r+1)T\lambda^{-n}|\p\Omega|,
\end{align}
where $|v\cdot n(x)|\leq  r+1$ due to the compact support of $\chi$ and $|v|\leq |v_*|+\lambda< r+1 $. Here $|\p\Omega|$ is the measure of the boundary $\p\Omega$.

The solutions to the linear equations are chosen as follows: 
$$
    V_1(v) = e^{-|v - v_*|^2} \hbox{ and }\quad  V_2 \equiv 1.
$$ 
Moreover, the boundary and initial data are taken as $g_j=V_j|_{\Gamma_-^T}$ and $h_j=V_j|_{t=0}$.
We then substitute such $V_j$ in the integral \eqref{DEF:S11} and, to simplify the expression, we define the function $P$ by 
\begin{align*}
    P(v,u,\omega) &:= V_{1}(u')V_{2}(v')+V_{1}(v')V_{2}(u') 
     - V_{1}(u)V_{2}(v)-V_{1}(v)V_{2}(u)\\
    &=e^{-|u'-v_*|^2} + e^{-|v'-v_*|^2} - e^{-|u-v_*|^2} - e^{-|v-v_*|^2}.
\end{align*}
Then straightforward computations give
$$
    \|P\|_{L^\infty(\R^{n}\times\R^n\times\S^{n-1})}\le 4\quad \hbox{ and }\quad \|\p_vP\|_{L^\infty(\R^{n}\times\R^n\times\S^{n-1})}\le C
$$
for some constant $C>0$. 
In particular, applying \eqref{incoming and outgoing velocities}, we obtain that when $v=v_*$,
\[
    P(v_*,u,\omega) = (1-e^{|(u-v_*)\cdot\omega|^2})(e^{-|(u-v_*)\cdot\omega|^2}-e^{-|u-v_*|^2}),
\] 
which indicates
$
    P(v_*,u,\omega)\le0
$ and 
$
 P(v_*,u,\omega) = 0 
$
only when  $\omega\perp(v_*-u)$ or $\omega = \pm \frac{v_*-u}{|v_*-u|}$.

\noindent{\bf Step 2: Upper bound of $\int \widetilde{S}dt$.} 
We consider the following estimate
\begin{equation}\label{eq_4_triangle}
\begin{aligned}
    &\quad \left|\int_\R \widetilde{S}(t,y_*+tv_*,v_*)\,dt \right| \\
    &\le  \left| \int_\R \int_{\R^n} \int_{\R^{n}} \varphi_\lambda(y,v)\widetilde{S}(t,y+tv,v)\, dy dv dt\right|\\
    &\quad + \left|\int_\R  \int_{\R^n} \int_{\R^{n}} \varphi_\lambda(y,v)\left(\widetilde{S}(t,y+tv,v) - \widetilde{S}(t,y_*+tv_*,v_*) \right)\, dy dv dt \right|=:I_1+I_2,
\end{aligned}
\end{equation}
where we used the fact that $\|\varphi_\lambda\|_{L^1(\R^n\times\R^n)}=1$. Notice that $\|g_j\|_{L^\infty(\Gamma^T_-)}\leq 1$ and $\|h_j\|_{L^\infty(U)}\leq 1$.
For the first term on the right-hand side, we then apply the above estimates \eqref{EST:phi lambda 1} - \eqref{EST:phi lambda 2} for $\psi_\lambda$ and Proposition~\ref{prop_4_1} to derive
\begin{align}\label{EST:I1}
    \begin{aligned}
        I_1\leq \frac{C}{\varepsilon_1\varepsilon_2} \delta (1+\lambda^{-n}) 
     + \frac{C}{\varepsilon_1\varepsilon_2} \LC \varepsilon_1 +\varepsilon_2\RC^3,
    \end{aligned}
\end{align} 
where $C$ depends on $r$, $\kappa$, $n$, $M$, $\Omega$, and $T$.  

To estimate $I_2$, its integral domain for $t$ variable can be shrunk back to $(0,T)$ since $(\Phi_1-\Phi_2)(\cdot,\cdot,|v|)$ is compactly supported in $(0,T)\times\Omega$. By making change of variables $y=y_*+\lambda \tilde y$ and $v=v_*+\lambda\tilde{v}$ and denoting 
$$
    z_*=y_*+tv_* \quad \hbox{ and }\quad \tilde{z} = \tilde{y}+t\tilde{v},
$$
we have 
\begin{align*}
    I_2&= \LV\int_0^T  \int_{B_1} \int_{B_1}
    \chi(\tilde{y})\chi(\tilde{v})
    \left(\widetilde{S}(t,z_*+\lambda\tilde{z},v_*+\lambda \tilde{v}) - \widetilde{S}(t,z_*,v_*) \right)\, d\tilde{y} d\tilde{v} dt \RV\\
    &=  \Big|\int_0^T  \int_{B_1} \int_{B_1}
    \chi(\tilde{y})\chi(\tilde{v})
    \Big(\int_{\R^n}\int_{\mathbb{S}^{n-1}} \Big[ \widetilde{K}(t,z_*+\lambda\tilde{z},v_*+\lambda \tilde{v},u,\omega)P(v_*+\lambda\tilde{v},u,\omega)\\
    &\hskip7cm- \widetilde{K}(t,z_* ,v_*,u,\omega)P(v_*,u,\omega)\Big] \,du d\omega \Big)\, d\tilde{y} d\tilde{v} dt \Big|\\ 
    &\leq  \int_0^T  \int_{B_1} \int_{B_1}
    \chi(\tilde{y})\chi(\tilde{v})
     \Xi_1 (t,\tilde{y},\tilde{v}) \,d\tilde{y}d\tilde{v} dt + \int_0^T \int_{B_1} \int_{B_1}
    \chi(\tilde{y})\chi(\tilde{v})
     \Xi_2 (t,\tilde{y},\tilde{v})\, d\tilde{y}d\tilde{v} dt,
\end{align*}
where we denote
$$
    \widetilde{K}:=K_1-K_2 = (\Phi_1-\Phi_2)\Psi.
$$
Moreover, the functions $\Xi_1$ and $\Xi_2$ are defined as follows: for $(t,\tilde{y},\tilde{v})\in (0,T)\times B_1\times B_1$,
\begin{align*}
    \Xi_1 (t,\tilde{y},\tilde{v})
     :=&  \int_{\R^n}\int_{\mathbb{S}^{n-1}} \LV\widetilde{K}(t,z_*+\lambda\tilde{z},v_*+\lambda \tilde{v},u,\omega)\RV \LV P(v_*+\lambda\tilde{v},u,\omega)-P(v_*,u,\omega)\RV \,dud\omega \\
    \leq &\, M \lambda \|\p_v P\|_{L^\infty (B_2\times\R^n\times\mathbb{S}^{n-1})},
\end{align*}
and
\begin{align*}
    \Xi_2 (t,\tilde{y},\tilde{v})
    :=& \int_{\R^n}\int_{\mathbb{S}^{n-1}} \LV \widetilde{K}(t,z_*+\lambda\tilde{z},v_*+\lambda \tilde{v},u,\omega)-\widetilde{K}(t,z_* ,v_*,u,\omega)\RV \LV P(v_*,u,\omega) \RV \,du d\omega
    \\
    \leq &\,C \lambda  \int_{\R^n}\int_{\mathbb{S}^{n-1}}|\nabla_v \widetilde{K} (t,z_*, v_*+ s_1\tilde{v},u,\omega)|  \,du d\omega
    \\
    &+C\lambda\int_{\R^n}\int_{\mathbb{S}^{n-1}}|\nabla_x \widetilde{K}(t,z_*+s_2\tilde z ,v_*,u,\omega)| \,dud\omega,
    \\
    \leq &\, C M \lambda  ,\quad 0<s_1,\,s_2< \lambda.
\end{align*}
Note that in the above estimates of $\Xi_1$ and $\Xi_2$, we used the fact that $|P|\leq 4$, the mean value theorem, and $\nabla_x^{\sigma_1} \nabla_v^{\sigma_2} \tilde{K}\in\mathcal{M}$ with integers $\sigma_j\geq 0$ satisfying $0\leq\sigma_1+\sigma_2\leq 1$. Combining these estimates together yields that
\begin{align}\label{EST:I2}
    I_2\leq C\lambda  , 
\end{align}
where $C$ depends on $M$ and $T$. 
Therefore, from \eqref{eq_4_triangle}, \eqref{EST:I1}, \eqref{EST:I2}, and $1<\lambda^{-n}$, by taking $\varepsilon_1=\varepsilon_2=:\hat\varepsilon$, we arrive at
\begin{equation}\label{eq_4_10}
\begin{aligned}
    \left|\int_\R \widetilde{S}(t,y_*+tv_*,v_*)dt \right| 
    &\leq 
    C \LC {1\over \varepsilon_1\varepsilon_2} \delta (1+\lambda^{-n}) 
    + \frac{1}{\varepsilon_1\varepsilon_2} \LC  \varepsilon_1 +\varepsilon_2 \RC^3   + \lambda\RC\\
    &\leq C m^{-1}\LC \hat\varepsilon^{-2} \lambda^{-n} m   \delta
    +\hat\varepsilon + \lambda\RC,
\end{aligned}
\end{equation}
where $m:=\LC{\kappa \over 8}\RC^{n+3}{1\over \Lambda}<1$ for some sufficiently large constant $\Lambda>1$. Thanks to this extra scaling $m$, it is clearer to see that $\hat\varepsilon$ is controlled by $\kappa$, see Remark~\ref{remark:small conditions} below for details.   

Now for $\delta\in (0,1)$, we then optimize the above estimate by finding the critical point of 
$$
    E_\delta(\hat\varepsilon,\lambda):=\hat\varepsilon^{-2} \lambda^{-n} m\delta
    +\hat\varepsilon + \lambda.
$$
A direct computation gives that
$$
    \p_{\hat \varepsilon}E_\delta = -2 \hat\varepsilon^{-3}\lambda^{-n}m  \delta+1=0,\quad \p_{\lambda}E_\delta = -n\hat\varepsilon^{-2}\lambda^{-n-1}m\delta+1=0,
$$
which leads to the critical points 
$$
    \hat\varepsilon = 2^{n+1\over n+3} n^{-{n\over n+3}}\LC{\kappa\over 8}\RC \LC{\delta\over \Lambda}\RC^{1\over n+3}\quad\hbox{ and }\quad 
    \lambda= 2^{-2\over n+3} n^{{3\over n+3}} \LC{\kappa\over 8}\RC \LC{\delta\over \Lambda}\RC^{1\over n+3}.
$$
Here $\lambda<1$ if $\kappa$ is small enough. 
Thus, substituting $\hat\varepsilon$ and $\lambda$ into the right-hand side of \eqref{eq_4_10} gives
\begin{equation}\label{eq_4_11}
\begin{aligned}
    \left|\int_\R \widetilde{S}(t,y_*+tv_*,v_*)dt \right| \leq C \delta^{1\over n+3},
\end{aligned}
\end{equation}
where $C$ depends on $n$, $\kappa$, $M$, $\Omega$, and $T$.

\noindent{\bf Step 3: Estimate of the light ray transform of $\widetilde\Phi$.} 
Recall that $\Psi$ is nonnegative and $P(v_*,u,\omega)\le 0$. Hence, we have $\Psi P(v_*,u,\omega)\leq 0$ in $\R^n\times\R^n\times\mathbb{S}^{n-1}$. 

For a fixed $|v_*|=r > 0$, we take $u\in B_1(v_*+2\hat{v}_*)$ with unit vector $\hat{v}_*=v_*/|v_*|$ and consider 
$$
    \omega\in \Theta:=\left\{\omega\in \mathbb{S}^{n-1}: 0< a\leq \omega \cdot {u-v_*\over |u-v_*|}\leq b <{1\over 3}\right\}.
$$
The following estimates follow directly:
\begin{enumerate}
	\item 
	$  
	1\leq |u-v_*|^2\leq 9,
	$ 
	
	\item 
	$  
	|(u-v_*)\cdot\omega|\geq |(u-v_*)|a\geq a
	$ 
	implies
	$ 
	1-e^{|(u-v_*)\cdot\omega|^2}\leq 1- e^{a^2} <0,
	$ 
	
	\item 
    $
    e^{-|(u-v_*)\cdot\omega|^2}-e^{-|u-v_*|^2}\geq e^{-9b^2}-e^{-1}>0.
	$
\end{enumerate}
They imply
$$
|P(v_*,u,\omega)|\geq (e^{a^2}-1) (e^{-9b^2}-e^{-1})\quad \hbox{for all }u\in B_1(v_*+2\hat{v}_*),\quad \omega\in \Theta,
$$
where the lower bound of $P$ is independent of $r$.
Therefore, for each $|v_*|=r$, we derive
	\begin{align}\label{EST:Psi}
		\int_{\R^n}\int_{\mathbb{S}^{n-1}} \Psi(v_*,u,\omega)|P(v_*,u,\omega)| \, d\omega du 
		& \geq \int_{B_1(v_*+2\hat{v}_*)}\int_{\Theta} \Psi(v_*,u,\omega)|P(v_*,u,\omega)| \, d\omega du \notag\\
		&\geq (e^{a^2}-1) (e^{-9b^2}-e^{-1})\int_{B_1(v_*+2\hat{v}_*)}\int_{\Theta} \Psi(v_*,u,\omega) \, d\omega du \notag\\
        &\geq c_0 |\Theta| |B_1(0)|(e^{a^2}-1) (e^{-9b^2}-e^{-1}),
	\end{align} 
by applying \eqref{CON:Psi}, that is, $\Psi(v_*,u,\omega)\geq c_0$ in $B_3(v_*)\times \mathbb{S}^{n-1}$.
Note that this estimate \eqref{EST:Psi} is uniform in $v_*\in \R^n$.
This yields that 
\begin{align*}
    \left|\int_\R \widetilde{S}(t,y_*+tv_*,v_*)dt \right| 
    &= \LV\int_\R \widetilde{\Phi}(t,y_*+tv_*,r) \, dt\RV  \LC \int_{\R^n}\int_{\mathbb{S}^{n-1}} \Psi(v_*,u,\omega)|P(v_*,u,\omega)| \, d\omega du \RC \\
    &\geq c_1 \LV L_r\widetilde\Phi(y_*,v_*,r) \RV,
\end{align*}
where the constant $c_1>0$ depends on $a$, $b$, $n$, $c_0$ and $\Theta$.
Finally, together with the upper bound of $\int \widetilde{S}\,dt$ in \eqref{eq_4_11}, we complete the proof.
\end{proof}

\begin{remark}\label{remark:small conditions}
With a priori constants $\kappa$ and $\delta$ in the forward problem, the choice of $\varepsilon_1$ and $\varepsilon_2$ above ensures that the data is in $\mathcal{X}_\kappa$ so that the well-posedness for the nonlinear equation hold. Indeed, for $0<\delta <1 < \Lambda$, it follows that
$$
    \hat\varepsilon = 2^{n+1\over n+3} n^{-{n\over n+3}}\LC{\kappa\over 8}\RC \LC{\delta\over \Lambda}\RC^{1\over n+3}< 2 \LC{\kappa\over 8}\RC ={\kappa \over 4}. 
$$
The chosen boundary and initial data for the linear equations in the above proof satisfy $\|g_j\|_{L^\infty(\Gamma^T_-)}\leq 1$ and $\|h_j\|_{L^\infty(U)}\leq 1$, which leads to
$$
    \|\varepsilon_1 g_1+\varepsilon_2g_2\|_{L^\infty(\Gamma^T_-)} + \|\varepsilon_1 h_1+\varepsilon_2 h_2\|_{L^\infty(U)} \leq  4\hat{\varepsilon}<\kappa.
$$
It follows that $(\varepsilon_1 g_1+\varepsilon_2g_2, \varepsilon_1 h_1+\varepsilon_2 h_2)\in \mathcal{X}_\kappa$ with $\varepsilon_j\geq 0$. 
\end{remark}

We denote the set $A_r$ by
$$
    A_r:=\{(\tau,\xi)\in \R\times \R^n: |\tau|\le r|\xi|\} 
$$
and define the Fourier transform of $q$ by
$$
    \hat{q}(\tau,\xi):=\int_{\R}\int_{\R^n} q(t,x)e^{-i(t,x)\cdot(\tau,\xi)}\,dx dt.
$$
In the next proposition, we show that the Fourier transform of $\widetilde\Phi(\cdot,\cdot,r)$ in $A_r$ is controlled by its scaled light ray transform.
Since $\Omega$ is bounded, we assume that $\Omega$ is contained in a ball $B_d$ with $d>0$.
\begin{prop}\label{prop:fourier transform kernel}
Under the same hypotheses of Theorem~\ref{thm:stability},  the Fourier transform of $\widetilde\Phi$ satisfies 
\begin{align*}
    \|\widehat{\widetilde\Phi}(\cdot,\cdot,r)\|_{L^\infty(A_r)}\leq C \delta^{1\over n+3},
\end{align*}
where the constant depends on $r$, $n$, $\kappa$,  $c_0$, $M$, $\Omega$, and $T$.
\end{prop}
\begin{proof}
    For a fixed $r>0$, for any point $(\tau,\xi)\in A_r$, we take the vector $v\in r\mathbb{S}^{n-1}$ to be
    $$
        v = - {\tau\over |\xi|^2}\xi + \LC r^2-{\tau^2\over |\xi|^2}\RC^{1/2}\zeta,
    $$ 
    where the vector $\zeta\in\mathbb{S}^{n-1}$ satisfies $\xi\cdot\zeta=0$. Then $\tau= - v\cdot\xi$ so that $|\tau|\leq r|\xi|$.
    Multiplying the scaled light ray transform of $\widetilde\Phi$ by $e^{-iy\cdot \xi}$ and integrating over $\R^n$ lead to
    \begin{align}\label{ID:light ray equivalence}
        \int_{\R^n} L_r \widetilde\Phi(y,v,r) e^{-iy\cdot \xi}\,dy 
        &= \int_{\R^n} \LC\int_\R \widetilde\Phi(s,y+sv,r)\,ds \RC e^{-iy\cdot \xi}\,dy \notag \\
        &= \int_\R\int_{\R^n} \widetilde\Phi(s,x,r) e^{-i(s,x) \cdot(-v\cdot\xi,\xi)}\,dxds \notag\\
        &= \widehat{\widetilde\Phi}(-v\cdot\xi,\xi,r)=\widehat{\widetilde\Phi}(\tau,\xi,r)
    \end{align}
    by applying the Fubini theorem and change of variable $x=y+sv$. Inheriting from the compact supportness of $\widetilde\Phi(\cdot,\cdot,r)$ in $(0,T)\times \Omega$, the scaled light ray transform $L_r\widetilde\Phi(\cdot, v,r)$ is supported in the ball $B_{d + T|v|}$. Here we consider the vector $v\in r\mathbb{S}^{n-1}$ and $B_{d+T|v|}= B_{d+rT}$.
    Therefore, deriving from \eqref{EST: light ray Phi} and \eqref{ID:light ray equivalence}, we obtain 
    \begin{align*}
        |\widehat{\widetilde\Phi}(\tau,\xi,r)|\leq  \int_{B_{d+rT}}| L_r \widetilde\Phi(y,v,r) |\,dy \leq C|B_{d+rT}|\delta^{1\over n+3}\quad \hbox{ for }\quad (\tau,\xi)\in A_r
        .
    \end{align*}
\end{proof}
 
For $\widehat{\widetilde\Phi}$ in the timelike region $\{(\tau,\xi)\in \R\times \R^n: |\tau|> r|\xi|\}$, we need the following result in \cite[Theorem~1]{Ve99}, see also \cite[Lemma 3.4] {bellassouedStableDeterminationOutside2017}, to control its behavior.
\begin{prop}\label{prop:extension Ve}
    Let $r_0,\, d_0>0$. Let $D\subset \R^{n+1}$ be an open, bounded and connected set such that $\{x\in D:\,d(x,\p D)>r\}$ is connected for any $r\in [0,r_0]$. Let $E\subset D$ be an open set such that $d(E,\p D)\geq d_0$. If $q$ is an analytic function with 
    $$
        \|\p^\beta q\|_{L^\infty(D)}\leq {\Pi |\beta|!\over \rho^{|\beta| }},
    $$
    where $\beta=(\beta_1,\ldots,\beta_{n+1})$ is a multiindex with nonnegative integer $\beta_k$ and $|\beta|=\sum^{n+1}_{k=1}\beta_k$.
    Then 
    \begin{align}\label{EST:analytic estimate}
    \|q\|_{L^\infty(D)}\leq (2\Pi)^{1-\tilde\mu(|E|/|D|)} \|q\|^{\tilde\mu(|E|/|D|)}_{L^\infty(E)},
    \end{align}
    where $\tilde\mu\in (0,1)$ depends on $d_0$, $D$, $n,\, r_0,\, \rho$, and $d(x,\p D)$. Here $d(\cdot,\cdot)$ is the distance function.  
\end{prop} 
In the proof below, we will take $D=B_2$ and $E$ to be an open subset of $B_1$ in Proposition~\ref{prop:extension Ve} so that \eqref{EST:analytic estimate} holds for all $x\in B_1$, which will be sufficient for our purpose.  
\subsection{Proof of Theorem~\ref{thm:stability}}
\begin{proof}[Proof of Theorem~\ref{thm:stability}]
    We first estimate the Fourier transform of $\widetilde\Phi$ in a low-frequency region $|(\tau,\xi)|\leq \alpha$ with the help of Proposition~\ref{prop:extension Ve} for $\alpha >0$. 
    
    For fixed $r$, we consider the function $q_\alpha(\tau,\xi) := \widehat{\widetilde\Phi}(\alpha(\tau,\xi),r)$ for $(\tau,\xi)\in\R^{n+1}$. Then $q_\alpha$ is analytic since $\widetilde\Phi$ is compactly supported. We deduce that
    \begin{align*}
        |\p^\beta q_\alpha(\tau,\xi)| 
        &= \LV \p^\beta \widehat{\widetilde\Phi}(\alpha(\tau,\xi),r)\RV= \LV(-i\alpha)^{|\beta|}   \int_{\R^{n+1}} (t,x)^\beta \widetilde\Phi(t,x,r)e^{-i\alpha (t,x)\cdot(\tau,\xi)}\,dxdt\RV\\
        &\leq \alpha^{|\beta|} (T^2+d^2)^{|\beta|\over 2} \int^T_0\int_{\Omega} |\widetilde\Phi(t,x,r)|\,dxdt
        \leq Ce^{\alpha} {|\beta|!\over \rho^{ |\beta|}},
    \end{align*}
    where we took $\rho=(T^2+d^2)^{-{1\over 2}}$ and used $\alpha^{|\beta|} \leq |\beta|! e^\alpha$. The constant $C>0$ depends on $M_1$, $\Omega$, and $T$.
    By Proposition~\ref{prop:extension Ve} with $\Pi= Ce^{\alpha}$ and $\mu=\tilde\mu(|E|/|B_2|)\in (0,1)$, we have
    \begin{align*}
        \|q_\alpha\|_{L^\infty (B_1)}\leq (2\Pi)^{1-\mu} \|q_\alpha\|^\mu_{L^\infty(E)},
    \end{align*}      
    where the set $A^{int}_r$ is the interior of $A_r$ defined by $A^{int}_r:=\{(\tau,\xi)\in \R\times \R^n: |\tau|< r|\xi|\}$ and $E=A_r^{int}\cap B_1$. With this estimate at hand, it follows that for any $(\tau,\xi)\in B_\alpha$,
    \begin{align}\label{EST:phi low frequency}
        \widehat{\widetilde\Phi}(\tau,\xi,r) = q_\alpha(\alpha^{-1}(\tau,\xi))\leq (2\Pi)^{1-\mu} \|q_\alpha\|^\mu_{L^\infty(E)} 
        \le C e^{\alpha({1-\mu})}
        \delta^{\mu\over n+3}.
    \end{align}
    Here the last inequality follows from Proposition~\ref{prop:fourier transform kernel} since $\alpha(\tau,\xi)\in A_r$ if $(\tau,\xi)\in E$, and $C$ depends on $r$, $n$, $\kappa$,  $c_0$, $M$, $M_1$, $\Omega$, and $T$.

    Next, we compute the $H^{-1}$-norm of $\widetilde\Phi$. Denote $z=(\tau,\xi)$ and $\langle z\rangle=(1+|z|^2)^{1/2}$. By Plancherel theorem and \eqref{EST:phi low frequency}, 
    \begin{align}\label{EST:stability proof}
        \|\widetilde\Phi(\cdot,\cdot,r)\|^2_{H^{-1}(\R^{n+1})} 
        &= \int_{|z|\leq \alpha} \langle z\rangle^{-2} |\widehat{\widetilde\Phi}(z,r)|^2\,dz + \int_{|z|> \alpha}\langle z\rangle^{-2} |\widehat{\widetilde\Phi}(z,r)|^2\,dz \notag\\
        &\leq C e^{(2-2\mu)\alpha}\delta^{2\mu\over n+3}\LC \int_{|z|\leq \alpha}  1\,dz\RC + \alpha^{-2} \|\widetilde\Phi(\cdot,\cdot,r)\|_{L^2(\R^{n+1})}  \notag\\
        &\leq C \LC \alpha^{n+1} e^{(2-2\mu)\alpha}\delta^{2\mu\over n+3} + \alpha^{-2}\RC  \notag\\
        &\leq C \LC e^{(3-2\mu)\alpha}\delta^{2\mu\over n+3} + \alpha^{-2} \RC,
    \end{align}
    where $C$ is independent of $\delta$. 
    There exists $0<\delta < \min\{1,\Lambda\}$ such that if \eqref{eq_stability_assumption} holds, then we take
    $$
        \alpha = {\mu \over (3-2\mu)(n+3)}|\log \delta|>0
    $$
    and substitute it into \eqref{EST:stability proof}. This implies
    \begin{align*} 
        \|\widetilde\Phi(\cdot,\cdot,r)\|_{H^{-1}(\R^{n+1})} 
        &\leq C \LC \delta^{\mu\over 2(n+3)} + |\log\delta|^{-1}\RC,
    \end{align*}
    with $C$ depends on $r$, $n$, $\kappa$,  $c_0$, $M$, $M_1$, 
    $\Omega$, and $T$.   

Lastly, we combine the above estimates together with Sobolev embedding theorem, that is,  $H^{[\frac{n+1}{2}]+1}((0,T)\times \Omega)	\hookrightarrow C^{0,\gamma}((0,T)\times \Omega)$ with some $\gamma>0$ depending on $n$ and then deduce
\[
\|\widetilde\Phi(\cdot,\cdot,r)\|_{L^\infty((0,T)\times \Omega)}\le C\|\widetilde\Phi(\cdot,\cdot,r)\|_{H^{[\frac{n+1}{2}]+1}((0,T)\times \Omega)},
\]
where $C$ depends on $n$, $T$, and $\Omega$, see \cite[Theorem 6, Section 5.6]{evansPartialDifferentialEquations2022}, and Sobolev interpolation theorem 
\[
\|\widetilde\Phi(\cdot,\cdot,r)\|_{H^{[\frac{n+1}{2}]+1}((0,T)\times \Omega)} \le C\|\widetilde\Phi(\cdot,\cdot,r)\|^\theta_{H^{-1}((0,T)\times \Omega)} \|\widetilde\Phi(\cdot,\cdot,r)\|^{1-\theta}_{H^{s}((0,T)\times \Omega)},
\]
where $C$ depends on $n$ and $0<\theta <1$ depends on $s$,
see \cite[Section 12.4, Chapter 1]{lionsNonHomogeneousBoundaryValue1972}. This completes the proof of Theorem \ref{thm:stability}.
\end{proof}

\section*{Acknowledgements}
The authors would like to express their gratitude to Professor Hanming Zhou for valuable discussions. R.-Y. Lai is partially supported by the National Science Foundation through grants DMS-2006731 and DMS-2306221.

\bibliographystyle{abbrv}
\bibliography{arXiv_230904}
\end{document}